\documentclass[10pt,a4paper]{article}
\usepackage[margin=1.3in]{geometry}
\usepackage[T1]{fontenc}
\usepackage{tikz}
\usetikzlibrary{positioning}
\usepackage[utf8]{inputenc}
\usepackage{graphicx}
\usepackage{mathtools}
\usepackage{amssymb}
\usepackage{amsthm}
\usepackage{thmtools}
\usepackage{xcolor}
\usepackage{nameref}
\usepackage{hyperref}
\usepackage{amsmath}
\usepackage{amsfonts}
\usepackage{wrapfig}
\usepackage[autostyle]{csquotes}
\usepackage{layout}
\usepackage{tabularx}
\usepackage{authblk}
\usetikzlibrary{shapes.geometric, positioning}
\usepackage{caption}
\usepackage{subcaption}
\usetikzlibrary{positioning,chains,fit,shapes,calc}
\usepackage{tikz}
\theoremstyle{definition}
\newtheorem{defn}{Definition}[section]
\newtheorem{theorem}{Theorem}[section]
\newtheorem{corollary}{Corollary}[theorem]
\newtheorem{lemma}[theorem]{Lemma}

\newtheorem{note}{Note}

\newtheorem{exmp}{Example}[section]
\theoremstyle{remark}
\newtheorem{remark}{Remark}[section]

\date{}
\begin{document}

	\title{\sc Intersection Hypergraph on $D_n$ }
	
	\author{{ \sc Sachin Ballal}\footnote{Corresponding Author}~~  and \sc{ Ardra A N} }
	\affil{School of Mathematics and Statistics, University of Hyderabad,
		500046, India\\ sachinballal@uohyd.ac.in, 23mmpp02@uohyd.ac.in}

	\date{}
	\maketitle
	
	\begin{abstract}
		Let $G$ be a group and $S$ be the set of all non-trivial proper subgroups of $G$. \textit{The intersection hypergraph of $G$}, denoted by $\tilde{\Gamma}_\mathcal{H}(G)$, is a hypergraph whose vertex set  is  $\{H \in S \,\, | \,\, H \cap K = \{e\} \,\, \text{for some} \,  K \in S  \}$ and  hyperedges are the maximal subsets of the vertex set with the property that any two vertices in it have a trivial intersection. The aim of this paper is to study the intersection hypergraph of dihedral groups, $\tilde{\Gamma}_\mathcal{H}(D_n)$. We examine some of the structural properties, viz., diameter, girth and chromatic number of $\tilde{\Gamma}_\mathcal{H}(D_n)$. Also, we provide characterizations for  hypertreees, star structures of  $\tilde{\Gamma}_\mathcal{H}(D_n)$, and investigate the planarity and non-planarity of  $\tilde{\Gamma}_\mathcal{H}(D_n)$.
	\end{abstract}
	\section{Introduction}	
	Hypergraph is a generalization of graph, allowing the analysis of multiple relationships rather than just pair-wise relationships. The notion of hypergraphs has been introduced by  C. Berge\cite{berge}. The study of hypergraphs on algebraic structures is an emerging area as a generalization of graph to extend some prominent results from graph theory.  In \cite{cameron2021graphs}, P. J. Cameron introduced different types of graphs on groups whose edges reflect the group structures in some way. D. F. Anderson and S. Al-Kaseasbeh \cite{anderson2023intersection} studied on the  completeness, bipartiteness and planarity of \textit{intersection subgroup graph of a group}. In \cite{devi2023complement}, P. Devi and R. Rajkumar studied on \textit{the complement of the intersection graph of subgroups of a group}. They characterized finite groups whose graphs are bipartite, tree, star and planar.   In \cite{eslahchi2007k}, Ch. Eslahchi and A. M. Rahimi introduced and studied \textit{k-zero divisor hypergraph} of a commutative ring. Later, K. Selvakumar \textit{et al.}\cite{amritha2022ideal,selvakumar2019genus,selvakumar2020k} introduced and studied on \textit{k-maximal hypergraph} and \textit{k-maximal ideal hypergraph} as the extension of the concepts of  co-maximal elements and co-maximal ideals of a ring. Also,  \textit{k-annihilating ideal hypergraph} of a ring was studied by  K. Selvakumar \textit{et al.}\cite{selvakumar2019genus1,selvakumar2019k}. Recently, in \cite{cameron2023hypergraphs}, P. J. Cameron \textit{et al.} introduced and studied four types of hypergraphs, namely commuting hypergaphs, power hypergraphs, enhanced power hypergraphs, and generating hypergraphs defined on groups. Motivated from the study of hypergraphs on algebraic structures, we have introduced and studied  hypergraphs on the subgroups of a group, in particular the subgroups of dihedral groups.

	A \textit{hypergraph} $\mathcal{H}$ is a pair $(V(\mathcal{H}),E(\mathcal{H}))$ of disjoint sets, where $V(\mathcal{H})$ is a non-empty finite set whose elements are called vertices and the elements of $E(\mathcal{H})$ are non-empty subsets of $V(\mathcal{H})$ called hyperedges. Any hypergraph $\mathcal{H}'=(V'(\mathcal{H}'),E'(\mathcal{H}'))$ such that $V'(\mathcal{H}')  \subseteq V(\mathcal{H})$ and $E'(\mathcal{H}') \subseteq E(\mathcal{H})$ is called a \textit{subhypergraph} of $\mathcal{H}$. The \textit{path} is a vertex–hyperedge alternative sequence, 
	where the vertex belongs to the consecutive hyperedge in the sequence. The \textit{cycle} 
	is a path whose first vertex is the same as the last vertex. The \textit{length} of a path is the
	number of hyperedges in the path. A hypergraph is \textit{connected} if there exists a path between any pair of vertices, otherwise 
	it is \textit{disconnected}. The \textit{distance} between two vertices is the minimum length of the 
	path connecting these two vertices. The \textit{diameter} of a hypergraph is the maximum 
	distance among all pairs of vertices.  The \textit{girth} of a hypergraph is the length of a shortest cycle it contains. A hypergraph is called a \textit{star} if there is a vertex which belongs to all hyperedges.  
	The \textit{incidence graph (or bipartite representation)} $\mathcal{I}(\mathcal{H})$ of $\mathcal{H}$ is a bipartite graph with vertex $V(\mathcal{H}) \cup E(\mathcal{H})$ and a vertex $v \in V(\mathcal{H}) $ is adjacent to a vertex $u \in E(\mathcal{H})$ iff $v \in u$ in $\mathcal{H}$.

	A \textit{proper vertex-coloring} (often simply called a proper coloring) of a hypergraph
	$\mathcal{H}$ is an assignment of colors to the vertices of $\mathcal{H}$ such that no hyperedge contains vertices all of the
	same color. The \textit{chromatic number} of $\mathcal{H}$, denoted by $\chi(\mathcal{H})$, is the minimum
	number of colors used by a proper vertex-coloring of $\mathcal{H}$.
	The \textit{chromatic index} of a hypergraph $\mathcal{H}$,  denoted by $\chi^{'}(\mathcal{H})$, is the minimum number of
	colors needed to color the edges of $\mathcal{H}$ such that no two intersecting edges have the same
	color.

	We denote by $S_n$ the surface obtained from the sphere $S_0$  by adding $n$ handles. The number $n$ is called the \textit{genus of the surface} $S_n, n\geq 0$.  The \textit{orientable genus} of a graph $G$, denoted by $g(G)$, is the minimum genus of
	a surface in which $G$ can be embedded.  A surface obtained by adding $k$ crosscaps to $S_0$ 
	is known as the non-orientable surface $N_k$. The number $k$ is called the crosscap of $N_k$. The non-orientable genus of
	a graph $G$, denoted by $\tilde{g}(G)$, is the smallest integer $k$ such that $G$ can be embedded on $N_k$. A planar
	graph is a graph of genus (crosscap) 0, a toroidal graph is a graph of
	genus 1, and a projective-planar graph is a graph of crosscap 1. Further, note that if $H$ is a subgraph of a graph $G$, then $g(H) \leq g(G)$ and $\tilde{g}(H) \leq \tilde{g}(G)$. For more details on graphs and hypergraphs, one may refer \cite{voloshinhypergraphs,white1985graphs}.

	In Section 2, we have introduced the \textit{intersection hypergraph of a group} and analyzed some of the structural properties, viz., diameter, girth, chromatic index and chromatic number of the intersection hypergraph  $\tilde{\Gamma}_\mathcal{H}(D_n)$ of $D_n$. Also, we have characterized hypertrees and star hypergraphs in terms of $n$. In Section 3, we have characterized planar and non-planar hypergraph of $D_n$.

	\section{Intersection hypergraph of $D_n$ and its structural properties}
	We begin this section with the definition of the intersection hypergraph on the subgroups of a group. Our focus is mainly on the study of the intersection hypergraphs  $\tilde{\Gamma}_\mathcal{H}(D_n)$ of dihedral groups. We have examined the structural properties of $\tilde{\Gamma}_\mathcal{H}(D_n)$ such as diameter, girth, chromatic index and chromatic number. Also, we have characterized the hypergraphs $\tilde{\Gamma}_\mathcal{H}(D_n)$ that are star hypergraphs and hypertrees. The following definition is motivated from \cite{cameron2023hypergraphs}.
	
	\begin{defn}
		Let $G$ be a group and $S$ be the set of all non-trivial proper subgroups of $G$. \textit{The intersection hypergraph of $G$}, denoted by $\tilde{\Gamma}_\mathcal{H}(G)$, is a hypergraph whose vertex set,\\ $V =\{H \in S \,\, | \,\, H \cap K = \{e\} \,\, \text{for some $ K \in S $} \}$ and $E \subseteq V$ is a hyperedge if and only if \begin{enumerate}
			\item For distinct $H,K \, \in E$, $H \cap K = \{e\}.$
			\item There does not exist  $E' \supset E$ which  satisfies (1).
		\end{enumerate}
		
	\end{defn}
	\begin{exmp}
		Consider the Klein-4 group, $$V_4=\{e,a,b,c \, | a^2=b^2=c^2=e, \,  ab=c=ba, \, ac=b=ca, \, bc=a=cb\}.$$ Then, the vertex set of $\tilde{\Gamma}_\mathcal{H}(V_4)$ is $ V= \{ \{e,a\}, \{e,b\}, \{e,c\}\}$ and the hyperedge set is $E= \{ \{\{e,a\}, \{e,b\}, \{e,c\}\}\}$. The hypergraph, $\tilde{\Gamma}_\mathcal{H}(V_4)$ is given in  Figure  \ref{klein4group}.
		
		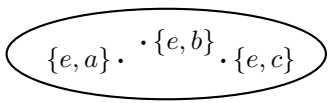
\begin{figure}[h]
			\centering
			\begin{tikzpicture}[scale=.3]
				\draw[thick] (0,0) ellipse (7cm and 2cm);
				
				\filldraw (-1cm, 0.5cm) circle (2pt) node[right] {$\{e,b\}$};    
				\filldraw (-2.0cm, -0.3cm) circle (2pt) node[ left] {$\{e,a\}$}; 
				\filldraw (2.5cm, -0.3cm) circle (2pt) node[right] {$\{e,c\}$};  
			\end{tikzpicture}
			\caption{$\tilde{\Gamma}_\mathcal{H}(V_4)$}
			\label{klein4group}
		\end{figure}

	\end{exmp}
	\begin{exmp}\label{d4eg}
		Consider the dihedral group of order 8, $D_4= <a,b\,|a^4=e=b^2, bab^{-1}=a^{-1}>$. The vertex set of $\tilde{\Gamma}_\mathcal{H}(D_4)$ is $ V=\{H_2,H_3,H_4,H_5, H_6,H_7,H_8,H_9\}$, where 
		$H_2=<a^2>, H_3=<b>,$ \\$ H_4=<ab>, H_5=<a^2b>, H_6=<a^3b>, H_7=<a>, H_8=<a^2,b>,$ $ H_9=<a^2,ab> $. The hyperedge set of 	$\tilde{\Gamma}_\mathcal{H}(D_4)$	is $E =\{e_1,e_2,e_3,e_4\}$, where  
		$e_1 = \{H_2,H_3,H_4,H_5,H_6\}$,
		$e_2=\{H_3,H_4,H_5,H_6,H_7\}$,
		$e_3=\{H_3,H_5,H_9\}$ and
		$e_4=\{H_4,H_6,H_8\}$. Figure \ref{d4} depicts the hypergraph  $\tilde{\Gamma}_\mathcal{H}(D_4)$ on $D_4$.
		
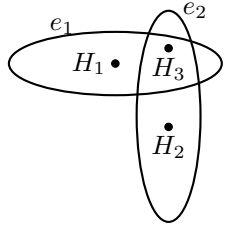
\begin{figure}[h]
			\centering
			
			\begin{tikzpicture}[scale=.8]
				
				\draw[thick] (0,0) ellipse (3 and .8);
				\draw[thick] (1.5,0) ellipse (3 and .8);
				\draw[thick] (-.5,-.5) ellipse (.8 and 2);
				\draw[thick] (1.5,-.5) ellipse (.8 and 2);

				\node at (-2, 1) {$e_1$};
				\node at (4, 1) {$e_2$};
				\node at (-.5, 1.7) {$e_3$};
				\node at (1.5, 1.7) {$e_4$};

				
				\filldraw (-2,-0.2cm) circle (2pt) node[above] {${H_2}$};
				\filldraw (-1,-0.2cm) circle (2pt) node[above] {${H_3}$};
				\filldraw (0,-0.2cm) circle (2pt) node[above] {${H_5}$};
				\filldraw (4, -0.2cm) circle (2pt) node[above] {${H_7}$};
				\filldraw (1, -0.2cm) circle (2pt) node[above] {${H_4}$};
				\filldraw (2, -0.2cm) circle (2pt) node[above] {${H_6}$};
				\filldraw (-0.5, -1cm) circle (2pt) node[below] {${H_9}$};
				\filldraw (1.5, -1cm) circle (2pt) node[below] {${H_8}$};

			\end{tikzpicture}
			\caption{$\tilde{\Gamma}_\mathcal{H}(D_4)$}
			\label{d4}
			
		\end{figure} 
	\end{exmp}\begin{flushright}
		$\square$
	\end{flushright}
	
	\noindent For a positive integer $n \geq 1$, the \textit{dihedral group of order $2n$}, denoted by $D_n$, is defined as
	$$D_n=<a,b\, | a^n=e, \, b^2=e,\,  bab^{-1}=a^{-1}>.$$ 
	
	\begin{theorem}\cite{conrad2009dihedral}\label{Kconrad1}
		Every subgroup of $D_n$ is cyclic or dihedral. A complete listing of the subgroups is as follows:
		\begin{enumerate}
			\item $< a^r > $, where $r \mid n$, with index 2$ r$,
			\item $ <a^r, a^i b> $, where $r \mid n $ and $0 \leq i \leq r - 1$, with index $r$.
		\end{enumerate}
		Every subgroup of $D_n$ occurs exactly once in this listing.
	\end{theorem}
	\begin{remark}
		\begin{enumerate}
			\item A subgroup of $D_n$ is said to be of \textbf{Type (1)} if it is cyclic as stated in (1) of Theorem \ref{Kconrad1}.
			\item A subgroup of $D_n$ is said to be of \textbf{Type (2)} if it is dihedral subgroup as stated in (2) of Theorem \ref{Kconrad1}.
		\end{enumerate}
	\end{remark}
	
	\begin{note}
		Since $\tilde{\Gamma}_\mathcal{H}(D_n)$ is empty for $n=1$, we exclude this case and consider $n \geq 2$ throughout the article.
	\end{note}

	\begin{lemma}
		$\tilde{\Gamma}_\mathcal{H}(D_n)$ is non empty and moreover, all the proper non-trivial subgroups constitute the vertex set.
	\end{lemma}
	\begin{proof}
		If $H_1$ is a non-trivial subgroup of $D_n$ of \textbf{Type (1)}, then choose  $ K_1 = <a^n, b> = \{e,b\} $. Hence, $H_1 \cap K_1 =\{e\}$ and therefore, $H_1 \, \in \,  V(\tilde{\Gamma}_\mathcal{H}(D_n))$.
		
		If $H_2$ is a proper subgroup of $D_n$ of \textbf{Type (2)}, then there exists  $j \in \{0,1,2,\cdots,n-1\}$ such that $a^j b \notin H_2$. For  $K_2 =<a^n, a^j b> = \{e,a^j b\}$, we have  $H_2 \cap K_2 =\{e\}$. Hence, $H_2 \, \in \, V(\tilde{\Gamma}_\mathcal{H}(D_n))$.
	\end{proof}
	In the following theorem, we established that  $\tilde{\Gamma}_\mathcal{H}(D_n)$ is connected and its diameter is atmost  3.
	
	\begin{theorem}\label{diameter}
		$\tilde{\Gamma}_\mathcal{H}(D_n)$ is connected  and $diam(\tilde{\Gamma}_\mathcal{H}(D_n)$) $\leq 3 $.
	\end{theorem}
	\begin{proof}
		Let $H_1, H_2$ be two distinct vertices of 	$\tilde{\Gamma}_\mathcal{H}(D_n)$. We prove that there exists a path from $H_1$ to $H_2$ and  $d(H_1, H_2) \leq 3$. \\
		If $H_1 \cap H_2 = \{e\}$, then there exists a hyperedge $e_1$ containing $H_1$ and $ H_2$. Thus, $H_1e_1H_2$ is a shortest path from $H_1 \, \text{to} \, H_2$ and  $d(H_1, H_2) =1$. 
		\begin{figure}[h]
			\centering
			\begin{tikzpicture}[scale=.7]
				
				\draw[thick] (0,0) ellipse (2 and .5);

				\node at (-1.75, .5) {$e_1$};

				
				\filldraw (-0.5, 0cm) circle (2pt) node[left] {${H_1}$};
				\filldraw (0.5, 0cm) circle (2pt) node[right] {${H_2}$};

			\end{tikzpicture}

		\end{figure}

		\noindent If $H_1 \cap H_2 \neq  \{e\}$, we have the following cases:
		
		\noindent\textbf{Case 1. Both $H_1$ and $H_2$ are of Type (1).}

		\noindent Consider $H_3 = <a^n, b> = \{e,b\}$. Observe that $H_1 \cap H_3 = \{e\}  \, \text{and}\,  H_2 \cap H_3 =\{e\}$. Therefore, there exist two distinct hyperedges $e_1 \, \text{and}\, e_2$ such that $e_1$ contains $H_1 \, \text{and} \, H_3$, and $e_2$ contains $H_2 \, \text{and} \, H_3$. Thus, $H_1 e_1 H_3 e_2 H_2$ is a shortest path from $H_1$ to $H_2$ and  $d(H_1, H_2) =2$.
		\begin{figure}[h]
			\centering
			\begin{tikzpicture}[scale=.7]
				
				\draw[thick] (-2,0) ellipse (2 and .6);
				
				\draw[thick] (-1,-1) ellipse (.6 and 2);
				
				\node at (-3, .7) {$e_1$};
				\node at (-.5, 1) {$e_2$};

				
				\filldraw (-2, 0cm) circle (2pt) node[left] {${H_1}$};
				\filldraw (-1, .29cm) circle (2pt) node[below] {${H_3}$};
				\filldraw (-1, -1.2cm) circle (2pt) node[below] {${H_2}$};
				
			\end{tikzpicture}

		\end{figure}
		
		\noindent\textbf{Case 2. One is of Type (1) and the other is of Type (2).}

		\noindent Choose $j \in \{0,1,2, \cdots, n-1\} $ such that $ a^jb \notin H_1 \cup H_2$. Consider $H_3 = <a^n,a^jb> = \{e, a^jb\}$. Observe that $H_1 \cap H_3 = \{e\}  \, \text{and}\,  H_2 \cap H_3 =\{e\}$. Therefore, there exist two distinct hyperedges $e_1 \, \text{and}\, e_2$ such that $e_1$ contains $H_1 \, \text{and} \, H_3$, and $e_2$ contains $H_2 \, \text{and} \, H_3$. Thus, $H_1 e_1 H_3 e_2 H_2$ is a shortest path from $H_1$ to $H_2$ and 
		$d(H_1, H_2) =2$.

		\noindent\textbf{Case 3. Both $H_1$ and $H_2$ are of Type (2).}
		
		Let $H_1 = <a^{r_1}, a^ib>$ and $H_2 = <a^{r_2}, a^jb>$, where $r_1 \neq 1, r_2 \neq 1, \,  r_1, r_2 \mid n  $ and $i,j$ are fixed positive integers such that $0 \leq i \leq r_1 - 1$ and $0 \leq j \leq r_2 -1$.
		
		\noindent Consider the following subcases:

		\noindent$\underline{\textbf{Subcase 3.1.}}$
		If $r_1 = r_2$,
		
		\noindent$\underline{\textbf{Subcase 3.1.(a)}}$	If there exists a positive integer $r_0$ such that $ r_0 < n, \, r_0 \mid n$ and  $<a^{r_0}> \cap$ \\  $<a^{r_1}>=\{e\}$, then consider $H_3 = <a^{r_0}>$. Hence,  $H_1 \cap H_3 = \{e\}  \, \text{and}\,  H_2 \cap H_3 =\{e\}$. Therefore, there exist two distinct hyperedges $e_1 \, \text{and}\, e_2$ such that $e_1$ contains $H_1 \, \text{and} \, H_3$, and $e_2$ contains $H_2 \, \text{and} \, H_3$. Thus, $H_1 e_1 H_3 e_2 H_2$ is a shortest path from $H_1$ to $H_2$ and $d(H_1, H_2) =2$.

		\noindent$\underline{\textbf{Subcase 3.1.(b)}}$ If there does not exists any $r_0$ such that $<a^{r_0}> \cap <a^{r_1}>=\{e\}$ , then consider the following:
		
		If there exists  $l \in \{0,1,2, \cdots, n-1\}$ such that $a^{l}b \notin H_1 \cup H_2$, then for $H_3 = <a^n, a^{l}b>=\{e,a^{l}b\}$, we have  $H_1 \cap H_3 = \{e\}  \, \text{and}\,  H_2 \cap H_3 =\{e\}$. Therefore, there exist two distinct hyperedges $e_1 \, \text{and}\, e_2$ such that $e_1$ contains $H_1 \, \text{and} \, H_3$, and $e_2$ contains $H_2 \, \text{and} \, H_3$. Thus, $H_1 e_1 H_3 e_2 H_2$ is a shortest path from $H_1$ to $H_2$ and $d(H_1, H_2) =2$.

		If such $l$ does not exist, then choose  $l_1$ and $l_2$ from the set  $ \{0,1,2, \cdots, n-1\}$  such that $a^{l_1}b \notin H_1$	but  $a^{l_1}b \in H_2$ and   $a^{l_2}b \notin H_2$ but $a^{l_2}b \in H_1$. Consider $K_1 = <a^n, a^{l_1}b>=\{e,a^{l_1}b\}$ and $K_2 =$ $ <a^n, a^{l_2}b>=\{e,a^{l_2}b\}$. Observe that $H_1 \cap K_1 = \{e\},  \, \, K_1 \cap K_2 =\{e\}$ and $  H_2 \cap K_2 =\{e\}$. Hence there exist distinct hyperedges $e_1, e_2, e_3$ such that $e_1$ contains $H_1 \, \text{and} \, K_1$, $e_2 \, \text{contains}\, K_1 \, \text{and} \, K_2$, and $e_3 \, \text{contains} \, H_2 \, \text{and}\, K_2$. Hence, $H_1e_1K_1e_2K_2e_3H_2$ is a shortest path from $H_1 \, \text{to}\, H_2$ and  $d(H_1, H_2) =3$.
		
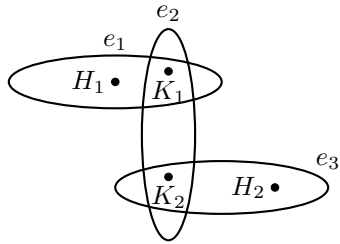
\begin{figure}[h]
			\centering
			\begin{tikzpicture}[scale=.7]
				
				\draw[thick] (-2,0) ellipse (2 and .5);
				\draw[thick] (0,-2) ellipse (2 and .5);
				\draw[thick] (-1,-1) ellipse (.5 and 2);
				
				\node at (-2, .75) {$e_1$};
				\node at (-1, 1.3) {$e_2$};
				\node at (2, -1.5) {$e_3$};
				
				
				\filldraw (-2, 0cm) circle (2pt) node[left] {${H_1}$};
				\filldraw (1, -2.0cm) circle (2pt) node[left] {${H_2}$};
				\filldraw (-1, -1.8cm) circle (2pt) node[below] {${K_2}$};
				\filldraw (-1, .2cm) circle (2pt) node[below] {${K_1}$};

			\end{tikzpicture}
			
		\end{figure}

		\noindent$\underline{\textbf{Subcase 3.2.}}$ If  $r_1 \neq r_2$,

		Choose  $l_1 \in \{0,1,2, \cdots , n-1\}$  such that $l_1 \not\equiv i (mod \, r_1)$ and $l_1 \not\equiv j (mod \, r_2)$ and so $a^{l_1} b \notin  H_1 \cup H_2 $. For $H_3= <a^n, a^{l_1} b>=\{e,a^{l_1} b\}$, we have  $H_1 \cap H_3 = \{e\}  \, \text{and}\,  H_2 \cap H_3 =\{e\}$. Therefore, there exist two distinct hyperedges $e_1 \, \text{and}\, e_2$ such that $e_1$ contains $H_1 \, \text{and} \, H_3$, and $e_2$ contains $H_2 \, \text{and} \, H_3$. Thus, $H_1 e_1 H_3 e_2 H_2$ is a shortest path from $H_1$ to $H_2$ and $d(H_1, H_2) =2$.
		
	\end{proof}
	
	\begin{corollary}\label{diametercoro}
		The diameter, $diam(\tilde{\Gamma}_\mathcal{H}(D_n)) = \begin{cases*}
			1, \, \text{$n$ is a prime},\\
			3, \, n=4k \,\,  \text{where $k$ is a positive integer},\\
			2, \, \text{otherwise}.
		\end{cases*}$
	\end{corollary}

	\begin{corollary}\label{diametercoro2}
		The chromatic index, $\chi^{'}(\tilde{\Gamma}_\mathcal{H}(D_n))= \begin{cases}
			m-1, \, n=4k \,\,  \text{where $k$ is a positive integer},\\
			m, \, \text{otherwise}.
		\end{cases}$\\
		where $m$ denotes the number of hyperedges of $\tilde{\Gamma}_\mathcal{H}(D_n)$.
	\end{corollary}
	
	\begin{proof}
		Let $e_1,e_2$ be two hyperedges of $\tilde{\Gamma}_\mathcal{H}(D_n)$. If $dist(e_1,e_2)$ is the minimum number of hyperedges in a sequence of incident hyperedges connecting $e_1$ and $e_2$, then $$diam(\tilde{\Gamma}_\mathcal{H}(D_n)) = max_{e_1,e_2 \in E} \, dist(e_1,e_2),$$ where $E$ is the hyperedge set of $\tilde{\Gamma}_\mathcal{H}(D_n).$
		Therefore, from the Corollary \ref{diametercoro}, 
		$$max_{e_1,e_2 \in E} \, dist(e_1,e_2) = \begin{cases*}
			1, \, \text{n is a prime},\\
			3, \, n=4k \,\,  \text{where $k$ is a positive integer},\\
			2, \, \text{otherwise}.
		\end{cases*}$$
		If $max_{e_1,e_2 \in E} \, dist(e_1,e_2) \leq 2$, then any two hyperedges in the hypergraph are incident. If  $n\neq 4k,$ then, by  Theorem \ref{diameter}, all the hyperedges of $\tilde{\Gamma}_\mathcal{H}(D_n)$ are incident. Hence,  $\chi^{'}(\tilde{\Gamma}_\mathcal{H}(D_n))=m$. Now, if $n= 4  k$ where $k$ is a positive integer, by the proof of Theorem \ref{diameter}, the hyperedge containing $<a^2,b>$ and the hyperedge  containing $<a^2,ab>$ are not incident and they are the only two non-incident hyperedges of the hypergraph $\tilde{\Gamma}_\mathcal{H}(D_n)$. Hence, $\chi^{'}(\tilde{\Gamma}_\mathcal{H}(D_n))=m-1$.

	\end{proof}
	
	In the following theorem, we have provided the necessary and sufficient conditions under which $\tilde{\Gamma}_\mathcal{H}(D_n)$ is a star hypergraph.
	\begin{theorem}\label{single hyperedge}
		For hypergraph  $\tilde{\Gamma}_\mathcal{H}(D_n)$ of $D_n$, the  following are equivalent:
		\begin{enumerate}
			\item $n$ is  prime.
			\item $\tilde{\Gamma}_\mathcal{H}(D_n)$ has exactly one hyperedge. 
			\item $\tilde{\Gamma}_\mathcal{H}(D_n)$ is  star.
			
		\end{enumerate}
	\end{theorem}
	\begin{proof}
		($1\Rightarrow2$).	Suppose that $n=p$, a prime number. Then, $ <a> $ is the only non-trivial subgroup of $D_n$ of \textbf{Type (1)} and
		the non-trivial proper subgroup of $D_n$ of \textbf{Type (2)} are  $ \{ <a^ib> \, | \, 0 \leq i \leq p-1\}$.
		
		\noindent Observe that, $<a> \cap <a^ib> =\{e\}\,  \, \text{for} \,  \, 0 \leq i \leq p-1$ and,  $<a^ib> \cap <a^jb> =\{e\}$ for all distinct $ i,j \in\{0,1,2,\cdots,p-1\}$. Therefore,  $\tilde{\Gamma}_\mathcal{H}(D_n)$ has exactly one hyperedge.

		($2 \Rightarrow3$). 	Suppose that $\tilde{\Gamma}_\mathcal{H}(D_n)$ has exactly one hyperedge. If  $\tilde{\Gamma}_\mathcal{H}(D_n)$ is not a star, then for each vertex $H_0$ of  $\tilde{\Gamma}_\mathcal{H}(D_n)$, there exists a hyperedge that does not contain $H_0$  which is a contradiction. Therefore,  $\tilde{\Gamma}_\mathcal{H}(D_n)$ must be a star.
		
		($3 \Rightarrow1$). Suppose that $\tilde{\Gamma}_\mathcal{H}(D_n)$ is a star. 
		If $n$ is not a prime, then we claim that for each vertex $H_0$ of $\tilde{\Gamma}_\mathcal{H}(D_n)$, there exists an another  vertex $H_1$ such that $H_0 \cap H_1 \neq \{e\}$.  \\ To prove the claim, consider the following cases:\\
		\noindent \textbf{ Case 1.} $n= p^\alpha$, for some prime $p$ and $\alpha \geq 2$.

		\noindent 	$\underline{\textbf{Subcase 1.1.}}$ Let $H_0$ be of \textbf{Type (1)}. Suppose $H_0 = <a^r>$, for fixed $r$ such that $r \mid n$ and $r\neq n $.
		
		\noindent 	$\underline{\textbf{Subcase 1.1.(a).}}$ If $r = 1$, then for $H_1 = <a^p>$,  $H_0 \cap H_1 = <a^p> \neq \{e\}$.
		
		\noindent 	$\underline{\textbf{Subcase 1.1.(b).}}$ If $r \neq 1$, then for $H_1 = <a>$, $H_0 \cap H_1 = <a^r> \neq \{e\}$.\\
		\noindent 	$\underline{\textbf{Subcase 1.2.}}$ Let $H_0$ be of \textbf{Type (2)}. Suppose $H_0 = <a^r,a^ib>$, for fixed $r$ such that $r \neq 1, \, r \mid n $ and  for fixed $i$ such that $ 0\leq i \leq r-1$.
		
		\noindent 	$\underline{\textbf{Subcase 1.2.(a).}}$ If $r=n$, then $H_0=<a^n,a^ib>=\{e,a^ib\}$. Now, choose  $l \in \{0,1,2,..., p-1\}$ such that $l \equiv i (mod \,  p)$ and consider $H_1= <a^p,a^lb>$. Hence, $H_0 \cap H_1=<a^ib> \neq \{e\}$.
		
		\noindent 	$\underline{\textbf{Subcase 1.2.(b).}}$ If $r \neq n$, then choose $H_1 = <a^n,a^ib>=\{e, a^ib\}$ and hence  $H_0 \cap H_1=<a^ib> \neq \{e\}$. \\
		
		\vspace{.02cm}
		\noindent 	\textbf{Case 2.} $n$ has at least two distinct prime divisors, i.e., $n = p_1p_2 \prod_i p_i^{\alpha_i}$, where $p_1,p_2$ are distinct primes and $\alpha_i$'s are non-negative integers.
		
		\noindent $\underline{\textbf{Subcase 2.1.}}$ Let $H_0$ be of \textbf{Type (1)}. Suppose $H_0 = <a^r>$, for fixed $r$ such that $r \mid n$ and $r\neq n $.
		\noindent $\underline{\textbf{Subcase 2.1.(a)}}$  If $r =1$, then for $H_1 = <a^{p_1}>$, $H_0 \cap H_1 = <a^{p_1}> \neq \{e\}$. 
		
		\noindent$\underline{\textbf{Subcase 2.1.(b)}}$ If $r \neq 1$, then for $H_1 = <a>$, $H_0 \cap H_1 = <a^r> \neq \{e\}$. \\
		$\underline{\textbf{Subcase 2.2.}}$ Let $H_0$ be of \textbf{Type (2)}.  Suppose $H_0 = <a^r,a^ib>$, for fixed $r$ such that $r \neq 1, \, r \mid n $ and  for fixed $i$ such that $ 0\leq i \leq r-1$.

		\noindent $\underline{\textbf{Subcase 2.2(a)}}.$If $r=n$, then $H_0=<a^n,a^ib>=\{e,a^ib\}$. Choose $l \in \{0,1,2,..., {p_1}-1\}$ such that $l \equiv i (mod \, p_1)$ and choose the subgroup $H_1= <a^{p_1},a^lb>$.
		
		\noindent $\underline{\textbf{Subcase 2.2.(b).}}$ If $r \neq n$, then choose $H_1 = <a^n,a^ib>=\{e, a^ib\}$ and hence  $H_0 \cap H_1=<a^ib> \neq \{e\}$. 
		\vspace{.001cm}

		\noindent From all the cases above, we have proved that for each vertex $H_0$ of $\tilde{\Gamma}_\mathcal{H}(D_n)$, there exists an another  vertex $H_1$ such that $H_0 \cap H_1 \neq \{e\}$.  So, there exists a hyperedge not containing $H_0$ which is a contradiction.
	\end{proof}

	We need the following definitions and results to characterize  hypergraphs $\tilde{\Gamma}_\mathcal{H}(D_n)$ of $D_n$ that are hypertrees in terms of $n$.
	\begin{defn}\cite{voloshinhypergraphs}
		A \textit{host graph} for a hypergraph is a connected graph on the same vertex set such that
		every hyperedge induces a connected subgraph of the host graph. A hypergraph $\mathcal{H} = (X,\mathcal{D})$ is called a \textit{hypertree} if there exists a host tree
		$T = (X,E)$ such that each edge $D \in \mathcal{D}$  induces a subtree in $T$. 
	\end{defn}
	\begin{defn}\cite{voloshinhypergraphs}
		A hypergraph $\mathcal{H}$ has the \textit{Helly property (is Helly, for short)} if for every subfamily of its edges the following implication holds:\\ If every two edges of the subfamily have a  non-empty intersection, then the whole subfamily has a non-empty intersection.
	\end{defn}
	\begin{defn}\cite{voloshinhypergraphs}
		The \textit{line graph} of a hypergraph $\mathcal{H}=(X,\mathcal{D})$, denoted by $\mathcal{L}(\mathcal{H})$, is the graph with the set $\mathcal{D}$ as the vertex set and two vertices are adjacent iff the respective edges intersect.
	\end{defn}
	\begin{defn}\cite{voloshinhypergraphs}
		A hypergraph is called \textit{chordal} if every cycle of length $\geq 4$ has two non-consecutive vertices which are adjacent. 
	\end{defn}
	
	\begin{lemma}\cite{voloshinhypergraphs}\label{hellyhypertree}
		Every hypertree is a Helly hypergraph.
	\end{lemma}
	\begin{lemma}\cite{voloshinhypergraphs}\label{hellychordal}
		A hypergraph $\mathcal{H}$ is a hypertree iff $\mathcal{H}$ is Helly and $\mathcal{L}(\mathcal{H})$ is chordal.
	\end{lemma}
	\begin{theorem}\label{hypertree}
		$\tilde{\Gamma}_\mathcal{H}(D_n)$ is a hypertree iff $n$ is a prime or $n = 4$.
	\end{theorem}
	\begin{proof}
		Suppose that $\tilde{\Gamma}_\mathcal{H}(D_n)$ is a hypertree. If  $n$ is neither a prime nor $n=4$, then consider the following cases:

		\noindent\textbf{Case 1.} $n = 4k \,\,  \text{for some positive integer $k \geq 2$}$.
		
		Consider $\mathcal{S}$ as the set of all hyperedges of  $\tilde{\Gamma}_\mathcal{H}(D_n)$ except the hyperedge $e_1$, where $e_1$  contains $<a^2,ab>$. Then, from the proof of Corollary \ref{diametercoro2}, $\mathcal{S}$ is a subfamily of  hyperedge set of  $\tilde{\Gamma}_\mathcal{H}(D_n)$ such that any two hyperedges in $\mathcal{S}$ have a non-empty intersection. Now, we claim that no vertex belongs to all the hyperedges of $\mathcal{S}$, i.e., for each vertex $H_0$, there exists a hyperedge $e_2 \in \mathcal{S}$ such that $H_0 \notin e_2$ and so $\tilde{\Gamma}_\mathcal{H}(D_{n})$ does not satisfy the Helly property. Therefore, by Lemma \ref{hellyhypertree},  $\tilde{\Gamma}_\mathcal{H}(D_n)$ is not a hypertree which is a contradiction.\\
		To prove the claim, consider the following cases:\\
		\textbf{Subcase 1.1.} Let $H_0$ be of \textbf{Type (1)}. Suppose $H_0 = <a^r>$, for fixed $r$ such that $r \mid n$ and $r\neq n $.\\
		$\underline{Subcase \, 1.1.(a).}$ If $r=1$, then consider the hyperedge $ e_2 = \{<b>, <ab>, \cdots, <a^{2^\alpha - 1}b>, \\ <a^2>\}.$
		Clearly, $e_2 \in \mathcal{S}$. But $H_0 \, \cap <a^2> =<a^2> \neq \{e\}$ and therefore,  $H_0 \notin e_2$.\\
		$\underline{Subcase \, 1.1.(b).}$ If $r \neq 1$, then consider the hyperedge $ e_2 = \{<b>, <ab>, \cdots, <a^{2^\alpha - 1}b>, \\ <a>\}.$ 	Clearly, $e_2 \in \mathcal{S}$. But $H_0 \, \cap <a> = <a^r> \neq \{e\}$ and therefore,  $H_0 \notin e_2$.\\
		\textbf{Subcase 1.2.}  Let $H_0$ be of \textbf{Type (2)}. Suppose $H_0 = <a^r,a^ib>$, for fixed $r$ such that $r \neq 1, \, r \mid n $ and  for fixed $i$ such that $ 0\leq i \leq r-1$.\\
		$\underline{Subcase \, 1.2.(a).}$ If $r =n$, then consider $e_2$ be the hyperedge containing $<a^4,a^lb>$, where  $ l \equiv i (mod \, 4)$. Clearly,  $e_2 \in \mathcal{S}$ but $H_0 \, \cap <a^4,a^lb>=<a^ib> \neq \{e\}.$ Therefore, $ H_0 \notin e_2$.\\
		$\underline{Subcase \, 1.2.(b).}$ If $r\neq n$, then choose $e_2$ be a hyperedge containing $<a^ib>$ in $ \mathcal{S}$. Observe that $H_0\, \cap$ $ <a^ib>=<a^ib> \neq \{e\}$ and therefore  $H_0 \notin e_2$. 
		
		\noindent\textbf{Case 2.} $n \neq 4k$ for any $k \geq 2$.
		
		\noindent By  Theorem \ref{diameter}, any two hyperedges of $\tilde{\Gamma}_\mathcal{H}(D_n)$ have a non-empty intersection. But, by  Theorem \ref{single hyperedge}, no vertex belongs to all the hyperedges of $\tilde{\Gamma}_\mathcal{H}(D_n)$. Thus, $\tilde{\Gamma}_\mathcal{H}(D_n)$  does not satisfy Helly property and so by Lemma \ref{hellyhypertree}, $\tilde{\Gamma}_\mathcal{H}(D_n)$ is not a hypertree which is a contradiction.

		Conversely, first suppose that $n$ is prime. Then, by Theorem \ref{single hyperedge}, $\tilde{\Gamma}_\mathcal{H}(D_n)$ has a single hyperedge consists of $p+1$ vertices, say $H_1,H_2, \cdots, H_{p+1}$.  Figure \ref{hosttree} depicts a host tree of $\tilde{\Gamma}_\mathcal{H}(D_n)$ and therefore $\tilde{\Gamma}_\mathcal{H}(D_n)$ is a hypertree.
		\begin{figure}[h]
			\centering
			\begin{tikzpicture}[scale=1.5, every node/.style={circle, fill=black, minimum size=6pt, inner sep=0pt}]
				
				\node[label=above:$H_1$] (u1) at (0,3) {};
				\node[label=above:$H_2$] (u2) at (.5,3) {};
				\node[label=above:$H_3$] (u3) at (1,3) {};
				
				\node[label=above:$H_p$] (u6) at (1.6,3) {};
				\node[label=above:$H_{p+1}$] (u7) at (2.3,3) {};
				
				\draw (u1) -- (u2);
				\draw (u2) -- (u3);
				\draw (u6) -- (u7);
				\draw[dashed] (1,3) -- (1.6,3);

			\end{tikzpicture}
			\caption{Host tree of $\tilde{\Gamma}_\mathcal{H}(D_n)$, when $n$ is a prime. }
			\label{hosttree}
		\end{figure}
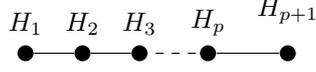
		
		Now, for $n=4$, we have the group $D_4$ as given in Example \ref{d4eg}. Clearly, $\tilde{\Gamma}_\mathcal{H}(D_4)$ is Helly. Now, the line graph $\mathcal{L}(\tilde{\Gamma}_\mathcal{H}(D_4))$ of $\tilde{\Gamma}_\mathcal{H}(D_4)$ as depicted in  Figure \ref{linegraph} is chordal. Therefore, by Lemma \ref{hellychordal},  $\tilde{\Gamma}_\mathcal{H}(D_4)$ is a hypertree.
		\begin{figure}[h]
			\centering
			\begin{tikzpicture}[scale=1.2, every node/.style={circle, fill=black, minimum size=6pt, inner sep=0pt}]
				
				\node[label=left:$e_1$] (u1) at (0,2) {};
				\node[label=left:$e_2$] (u2) at (0,1) {};
				\node[label=right:$e_3$] (v1) at (2,2) {};
				\node[label=right:$e_4$] (v2) at (2,1) {};
				
				\draw (u1) -- (u2);
				\draw (u1) -- (v1);
				\draw (u1) -- (v2);
				\draw (u2) -- (v1);
				\draw (u2) -- (v2);
				
			\end{tikzpicture}
			\caption{ $\mathcal{L}(\tilde{\Gamma}_\mathcal{H}(D_4))$}
			\label{linegraph}
		\end{figure}
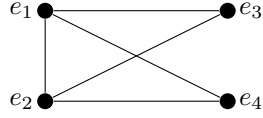
	\end{proof}
	
	In the following theorem, we find the girth of the hypergraph, $\tilde{\Gamma}_\mathcal{H}(D_n)$.
	\begin{theorem}
		$ gr(\tilde{\Gamma}_\mathcal{H}(D_n)) \in \{2, \infty\}$.
	\end{theorem}
	\begin{proof}
		Consider the following cases.
		
		\textbf{Case 1.} If $n = p$, a prime number, then by Theorem \ref{single hyperedge}, $\tilde{\Gamma}_\mathcal{H}(D_n)$ has exactly one hyperedge. Thus, the hypergraph does not contain a cycle and so
		$ gr(\tilde{\Gamma}_\mathcal{H}(D_n))=\infty$. 
		
		\textbf{Case 2.} If $n$ is not a prime, then  consider  $H_1 = \{e,b\} $ and $ H_2=\{e,ab\}$. Clearly, $H_1 \, \cap \, H_2 =\{e\}$. We claim that there exist two distinct hyperedges $e_1$ and $e_2$ such that both contains $H_1$ and $H_2$, and so $H_1e_1H_2e_2H_1$ is a shortest cycle of length 2. Hence, $  gr(\tilde{\Gamma}_\mathcal{H}(D_n))=2$. 
		
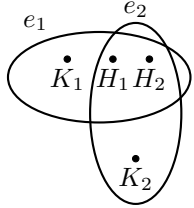
\begin{figure}[h]
			\centering
			\begin{tikzpicture}[scale=.6]
				
				\draw[thick] (-1.8,.8) ellipse (2 and 1);
				
				\draw[thick] (-1,0) ellipse (1 and 2);
				
				\node at (-3.2, 2) {$e_1$};
				\node at (-1, 2.3) {$e_2$};


				\filldraw (-1.5, 1.2cm) circle (2pt) node[below] {${H_1}$};
				\filldraw (-.7, 1.2cm) circle (2pt) node[below] {${H_2}$};
				\filldraw (-1, -1cm) circle (2pt) node[below] {${K_2}$};
				\filldraw (-2.5, 1.2cm) circle (2pt) node[below] {${K_1}$};
			\end{tikzpicture}
		\end{figure}\\
		Consider the following cases:
		
		\noindent $\underline{\text{Subcase}\,\, 2.1}\quad  \text{Suppose} \,n=p^\alpha, \, \text{for some prime $p$} \, \text{and} \, \alpha \geq 2$. Then, consider  $K_1 =<a^p>$ and \\ $K_2 =<a>$. Observe that $H_i \cap K_j =\{e\}$ for all $i,j \in \{1,2\}$ and $ K_1 \cap K_2 \neq \{e\}$. Therefore, there exist two distinct hyperedges $e_1, e_2$ such that $e_1$ contains $H_1, H_2\, \text{and} \, K_1$, and $e_2$ contains $H_1, H_2 \, \text{and} \, K_2$.\\ 
		
		\noindent $\underline{\text{Subcase}\,\, 2.2}\quad \text{Suppose} \, n$ contains atleast two distinct prime divisors. i.e.,  $n=p_1 p_2 \prod_{i} p_i^{\alpha_i}, $where $p_1,p_2$ are distinct primes and $\alpha_i$'s are non-negative integers.
		
		\noindent Consider $K_1 = <a^{p_1}> \, \text{and} \,  K_2 =<a>$. Observe that $H_i \cap K_j =\{e\}$ for all $i,j \in \{1,2\}$ and $ K_1 \cap K_2 \neq \{e\}$. Therefore, there exists two distinct hyperedges $e_1, e_2$ such that $e_1$ contains $H_1, H_2$ and $ K_1$, and $e_2$ contains $H_1, H_2, \, \text{and} \, K_2$.
	\end{proof}

	The following result discusses the chromatic number of $\tilde{\Gamma}_\mathcal{H}(D_n)$.
	\begin{theorem}
		The chromatic number,
		$\chi(\tilde{\Gamma}_\mathcal{H}(D_n)) = 2.$
	\end{theorem}
	\begin{proof}
		Divide the vertex set of $\tilde{\Gamma}_\mathcal{H}(D_n)$ into two sets $A \, \text{and} \, B$, where\\
		$A = \{<a^ib>, <a^d,a^jb> \, | \,   i \in \{1,2, \cdots, n-1\}, \, d \mid n, \,  \text{and} \, 1 \leq j \leq d-1\}$ and \\$ B = \{<a^d>, <b>, <a^d,b> \, | \,  d \mid n\}$.

		We claim that every hyperedge of $\tilde{\Gamma}_\mathcal{H}(D_n)$  has at least one vertex from  A as well as B.
		Let $e_1$ be a hyperedge of $\tilde{\Gamma}_\mathcal{H}(D_n)$. If $e_1$ contains only vertices from $B$, then we can find a vertex $<a^ib>$ in $A$ for some $i \in \{1,2, \cdots, n-1\}$ such that $<a^ib> \cap \,\,H =\{e\}$, for all  $H \in e_1$.
		Hence, the vertex $<a^ib>$ must belongs to $e_1$ and this contradicts the maximality of $e_1$. Similarly, the hyperedge cannot contain only vertices from the set $A$. Hence, each hyperedge has at least one vertex from both A and B. Assign the color $c_1$ to the vertices in $A$ and the color $c_2$ to the vertices in $B$. This is a proper coloring of $\tilde{\Gamma}_\mathcal{H}(D_n)$. Consequently,  
		$\chi(\tilde{\Gamma}_\mathcal{H}(D_n)) = 2.$
	\end{proof}
	
	\section{Planarity and Non-Planarity of  hypergraph}
	To discuss the planarity and orientable genus  of  $\tilde{\Gamma}_\mathcal{H}(D_n)$ on $D_n$, we need the following results.

	\begin{theorem}\cite{walsh1975hypermaps}
		A graph $G$ is planar iff it contains no subdivision of $K_5$ or $K_{3,3}$.
	\end{theorem}
	\begin{theorem}\cite{walsh1975hypermaps} 
		A hypergraph is planar iff its incidence graph is planar.
	\end{theorem}
	
	\begin{lemma}\label{genus}\cite{white1985graphs}
		If $G$ is a graph with $n$ vertices, $m$ edges, girth $gr$ and genus $g$, then $$ \dfrac{m(gr-2)}{2 \cdot gr}- \dfrac{n}{2} + 1 \leq g $$
	\end{lemma}
	\begin{theorem}\label{incidencegraph}\cite{walsh1975hypermaps}
		For any hypergraph $\mathcal{H}$, $g(\mathcal{H})=g(\mathcal{I}(\mathcal{H}))$.
	\end{theorem}
	
	The following two results characterize planarity and non-planarity of the hypergraph $\tilde{\Gamma}_\mathcal{H}(D_n)$ on $D_n$. We characterize the hypergraph $\tilde{\Gamma}_\mathcal{H}(D_n)$ which are planar in terms of $n$. Also, we characterize the non-planar hypergraph $\tilde{\Gamma}_\mathcal{H}(D_n)$ in terms of genus of its incidence graph.
	\begin{theorem}\label{planar}
		$\tilde{\Gamma}_\mathcal{H}(D_n)$ is planar iff $n$ is a prime or $n =4$.

	\end{theorem}
	\begin{proof} Suppose that $\tilde{\Gamma}_\mathcal{H}(D_n)$ is planar. Assume $n$ is neither a prime nor $n=4$. We will find three distinct vertices $H_1,H_2,H_3$ of $\tilde{\Gamma}_\mathcal{H}(D_n)$ and three distinct hyperedges $e_1,e_2,e_3$ of  $\tilde{\Gamma}_\mathcal{H}(D_n)$ such that all $H_1,H_2,H_3$ belongs to each $e_1,e_2,e_3$. Consider the  subhypergraph $G$ of $\tilde{\Gamma}_\mathcal{H}(D_n)$ with  $  \{e_1,e_2,e_3\}$ as hyperedge set and  the set of all vertices in $e_1,e_2,\, \text{and} \,\, e_3$ as the vertex set of $G$. So, the incidence graph $\mathcal{I}(G)$ of G  contain $K_{3,3}$. Hence, the incidence graph $\mathcal{I}(\tilde{\Gamma}_\mathcal{H}(D_n))$ of  $\tilde{\Gamma}_\mathcal{H}(D_n)$, contains $K_{3,3}$ and so 	$\tilde{\Gamma}_\mathcal{H}(D_n)$ is non-planar, a contradiction.
		
		
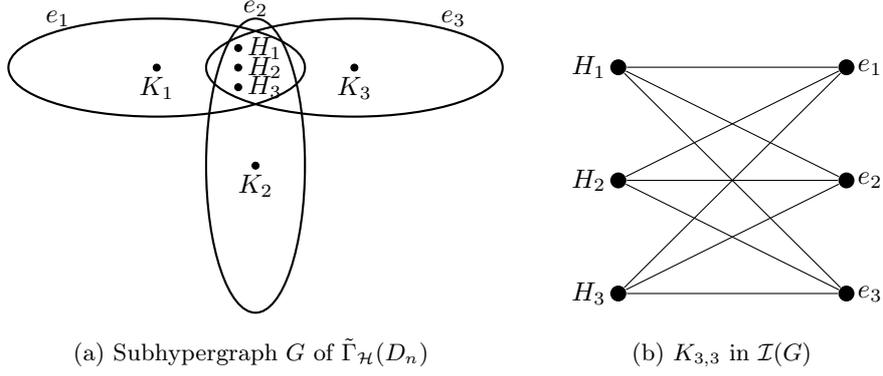
\begin{figure}[h]
			\centering\subfloat[Subhypergraph $G$ of $\tilde{\Gamma}_\mathcal{H}(D_n)$]{	\begin{tikzpicture}[scale=.65]
					
					\draw[thick] (-2,0) ellipse (3 and 1);
					\draw[thick] (2,0) ellipse (3 and 1);
					\draw[thick] (0,-2) ellipse (1 and 3);
					
					\node at (-4, 1) {$e_1$};
					\node at (0, 1.2) {$e_2$};
					\node at (4, 1) {$e_3$};
					
					
					\filldraw (-0.35, 0.4cm) circle (2pt) node[right] {${H_1}$};
					\filldraw (-0.35, 0cm) circle (2pt) node[right] {${H_2}$};
					\filldraw (-0.35, -.4cm) circle (2pt) node[right] {${H_3}$};
					\filldraw (0, -2cm) circle (2pt) node[below] {${K_2}$};
					\filldraw (-2, 0cm) circle (2pt) node[below] {${K_1}$};
					\filldraw (2, 0cm) circle (2pt) node[below] {${K_3}$};
					
			\end{tikzpicture}}
			\label{G}
			\hspace{.50 cm}
			\subfloat[$K_{3,3}$ in $\mathcal{I}(G)$]{
				\begin{tikzpicture}[scale=1.5, every node/.style={circle, fill=black, minimum size=6pt, inner sep=0pt}]
					
					\node[label=left:$H_1$] (u1) at (0,3) {};
					\node[label=left:$H_2$] (u2) at (0,2) {};
					\node[label=left:$H_3$] (u3) at (0,1) {};
					
					\node[label=right:$e_1$] (v1) at (2,3) {};
					\node[label=right:$e_2$] (v2) at (2,2) {};
					\node[label=right:$e_3$] (v3) at (2,1) {};
					
					\draw (u1) -- (v1);
					\draw (u1) -- (v2);
					\draw (u1) -- (v3);
					\draw (u2) -- (v1);
					\draw (u2) -- (v2);
					\draw (u2) -- (v3);
					\draw (u3) -- (v3);
					\draw (u3) -- (v2);
					\draw (u3) -- (v1);
					
			\end{tikzpicture}}
			
			\label{K33}

		\end{figure}
		\noindent Choose  $ K_1=<a>, K_2=<a^p>\, \text{and}\, K_3=<a^p,b>$,  where $p$ is a prime divisor of $n$. Also, choose positive integers $l_1,l_2,l_3$ from the set $ \{0,1,2,\cdots,n-1\}$ such that $ l_i \not\equiv 0 (mod \, p)$. So, $a^{l_1}b,a^{l_2}b,a^{l_3}b\notin K_3$. Consider $H_1 =$ $<a^{l_1}b>, H_2=<a^{l_2}b> \, \text{and} \, H_3=<a^{l_3}b>$. Thus,  $H_i \cap H_j=\{e\}$ for  $i\neq j \, \text{and} \, i,j \in \{1,2,3\}.$ Also, $  K_i \cap K_j \neq \{e\}$ and $ H_i \cap K_j = \{e\} \, \, \text{for all} \, \, i,j \in \{1,2,3\}.$ Hence, there exists three distinct hyperedges $e_1,e_2,e_3$ such that $e_1$ contains $H_1,H_2,H_3$ and $K_1$, $e_2$ contains $H_1,H_2,H_3$ and $K_2$, and $e_3$ contains $H_1,H_2,H_3$ and $K_3$.

		Conversely, first suppose that $n$ is prime. By Theorem \ref{single hyperedge}, $\tilde{\Gamma}_\mathcal{H}(D_n)$ has a single hyperedge. So, the corresponding incidence graph can be embedded on a plane and hence, $\tilde{\Gamma}_\mathcal{H}(D_n)$	is planar.

		Now, for $n=4$, we have the group $D_4$ as given in Example \ref{d4eg}.
		The incidence graph $\mathcal{I}(\tilde{\Gamma}_\mathcal{H}(D_4))$ of $\tilde{\Gamma}_\mathcal{H}(D_4)$ is depicted in Figure  \ref{IncidenceGraph}. 	Clearly, $\mathcal{I}(\tilde{\Gamma}_\mathcal{H}(D_4))$ can be embedded on a plane. So, $\tilde{\Gamma}_\mathcal{H}(D_4)$ is planar.
		\begin{figure}[h]
			\centering
			\begin{tikzpicture}[scale=.35, every node/.style={circle, fill=black, minimum size=6pt, inner sep=0pt}]
				
				\node[label=left:$H_2$] (f1) at (-1,2) {};
				\node[label=left:$H_3$] (f2) at (-1,1) {};
				\node[label=left:$H_4$] (f3) at (-1,0) {};
				\node[label=left:$H_5$] (f4) at (-1,-1) {};
				\node[label=left:$H_6$] (f5) at (-1,-2) {};
				\node[label=left:$H_7$] (f6) at (-1,-3) {};
				\node[label=left:$H_8$] (f7) at (-1,-4) {};
				\node[label=left:$H_9$] (f8) at (-1,-5) {};

				\node[label=right:$e_1$] (s9) at (6,1) {};
				\node[label=right:$e_2$] (s10) at (6,0) {};
				\node[label=right:$e_3$] (s11) at (6,-1) {};
				\node[label=right:$e_4$] (s12) at (6,-2) {};

				\draw (f1) -- (s9);
				\draw (f2) -- (s9);
				\draw (f2) -- (s10);
				\draw (f2) -- (s11);
				\draw (f3) -- (s9);
				\draw (f3) -- (s10);
				\draw (f3) -- (s12);
				\draw (f4) -- (s9);
				\draw (f4) -- (s10);
				\draw (f4) -- (s11);
				\draw (f5) -- (s9);
				\draw (f5) -- (s10);
				\draw (f5) -- (s12);
				\draw (f6) -- (s10);
				\draw (f7) -- (s12);
				\draw (f8) -- (s11);
				
			\end{tikzpicture}
			\caption{ $\mathcal{I}(\tilde{\Gamma}_\mathcal{H}(D_4))$}
			\label{IncidenceGraph}
		\end{figure}
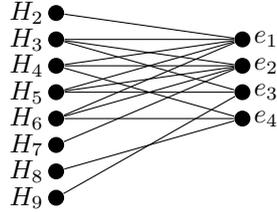
	\end{proof}
	
	\begin{theorem}
		$\tilde{\Gamma}_\mathcal{H}(D_n)$ is non-planar iff the orientable genus,	$g(\tilde{\Gamma}_\mathcal{H}(D_n))$ is at least 2.
	\end{theorem}
	\begin{proof}
		It is clear that if the genus $g(\tilde{\Gamma}_\mathcal{H}(D_n))$ is at least two, then $\tilde{\Gamma}_\mathcal{H}(D_n)$ is non-planar. 
		
		Conversely, suppose that $\tilde{\Gamma}_\mathcal{H}(D_n)$ is non-planar. Now, by Theorem \ref{planar},  $\tilde{\Gamma}_\mathcal{H}(D_n)$ is non-planar iff $n$ is neither a prime nor  $ n= 4$. To prove that $g(\tilde{\Gamma}_\mathcal{H}(D_n))$ is atleast 2, we consider the following cases:

		\noindent \textbf{Case 1}. $n$ is even.

		\noindent $\underline{\textbf{Subcase 1.1}}$ Suppose $n=6$.
		The non-trivial proper subgroups of $D_6$ are as follows:\\
		$H_2 =<a^2>, H_3=<a^3>, H_4=\{e,b\}, H_5=\{e,a^3b\}, H_6=\{e,ab\}, H_7=\{e,a^4b\}, H_8=\{e,a ^2b\}, \\ H_9=\{e,a^3b\}, H_{10}=<a^2,ab>, H_{11}=<a>, H_{12}=<a^2,b>, H_{13}=<a^3,b>, H_{14}=<a^3,ab>,\\$ $ H_{15}=<a^3,a^2b>$. So, the hyperedges of $\tilde{\Gamma}_\mathcal{H}(D_6)$ are:\\
		$e_1=\{H_2,H_3,H_4,H_5,H_6,H_7,H_8,H_9\}$,
		$e_2=\{H_2,H_4,H_5,H_8,H_9,H_{14}\}$,
		$e_3=\{H_2,H_4,H_5,H_6,H_7,H_{15}\}$,
		$e_4=\{H_2,H_6,H_7,H_8,H_9,H_{13}\}$,
		$e_5=\{H_3,H_4,H_7,H_8,H_{10}\}$,
		$e_6=\{H_3,H_5,H_6,H_9,H_{12}\}$ and\\
		$e_7=\{H_4,H_5,H_6,H_7,H_8,H_9, H_{11}\}$.\\
		In the incidence graph $\mathcal{I}(\tilde{\Gamma}_\mathcal{H}(D_6))$ of $\tilde{\Gamma}_\mathcal{H}(D_6)$,  the number of vertices, $n_0 = 21$, the number of edges, $m=43$ and the girth,  $gr= 4$, as $H_2e_1H_4e_2H_2$  is a shortest cycle of length 4 in $\mathcal{I}(\tilde{\Gamma}_\mathcal{H}(D_6))$.\\ Thus, by Lemma \ref{genus}, the genus $
		g(\mathcal{I}(\tilde{\Gamma}_\mathcal{H}(D_6)))  \geq \dfrac{m(gr -2)}{2 \cdot gr}-\dfrac{n_0}{2} + 1 
		= \dfrac{43 \cdot 2}{2 \cdot 4} - \dfrac{21}{2}+ 1 
		= 1.25.$
		
		\noindent Consequently, by Theorem \ref{incidencegraph} $g(\tilde{\Gamma}_\mathcal{H}(D_6))$ is at least 2.

		\noindent	$\underline{\textbf{Subcase 1.2}}$ Suppose $n \geq 8$.
		Let $G$ be the subhypergraph $\tilde{\Gamma}_\mathcal{H}(D_{n})$ induced by the vertex set, \\
		$ \Omega=\{H_0,H_1,...,H_{n-1}, <a>, <a^{\frac{n}{2}}>, <a^{\frac{n}{2}},b>, <a^{\frac{n}{2}},ab>,<a^{\frac{n}{2}},a^2b>\},$
		where \\ $  H_i = <a^ib>,  \, i \in \{0,1,...,n -1\} $. 
		Hence, 
		$e_1 = \{H_0,H_1,...,H_{n-1},<a>\}$,\\
		$e_2 = \{H_0,H_1,...,H_{n-1},<a^{\frac{n}{2}}>\}$,
		$e_3 = \{H_1,H_2...,H_{n-1},<a^{\frac{n}{2}},b>\}$, \\
		$e_4 = \{H_0,H_2,...,H_{n-1},<a^{\frac{n}{2}},ab>\}$ and
		$e_5 = \{H_0,H_1,...,H_{n-1},<a^{\frac{n}{2}},a^2b>\}$
		are the only hyperedges of the subhypergraph $G$ induced by $\Omega$.
		Consider the corresponding incidence graph $\mathcal{I}(G)$ of $G$. In $\mathcal{I}(G)$, the number of vertices, $n_0= (n+5)+5, $
		the number of edges, $m = 5n-1,$ and 
		the girth, $gr = 4$. So, by Lemma \ref{genus} and as $n \geq 8$,
		\begin{align*}
			g(\mathcal{I}(\tilde{\Gamma}_\mathcal{H}(D_n)))  \geq \dfrac{m\cdot (gr - 2)}{2 \cdot gr} - \dfrac{n_0}{2} + 1 & \geq \dfrac{(5n-1)(4 - 2)}{2 \cdot 4} -\dfrac{n+10}{2} + 1  \\
			& = \dfrac{3n-17}{4}\\
			& \geq 1.75. 
		\end{align*}
		Consequently, by Theorem \ref{incidencegraph},  $g(\tilde{\Gamma}_\mathcal{H}(D_n))$ is at least 2.
		
		\noindent \textbf{Case 2}. $n$ is odd.
		
		\noindent $\underline{\textbf{Subcase 2.1}}$: $n = p^\alpha, \, p \, \, \text{is an odd prime} \, \text{and} \,   \alpha \geq 2 .$

		\noindent $\underline{\textbf{Subcase 2.1(a)}}:$ Suppose $p=3 \, \text{and} \,  \alpha=2, \,\, \text{i.e.,} \,  n = 3^2= 9$. 
		The non-trivial proper subgroups of $D_9$ are as follows:
		$H_2 =<a^3>, H_3=<a>,  H_4=\{e,b\}, H_5=\{e,ab\}, H_6=\{e,a^2b\}, H_7=\{e,a^3b\}, \\ H_8=\{e,a ^4b\}, H_9=\{e,a^5b\}, H_{10}=\{e,a^6b\}, H_{11}=\{e,a^7b\}, H_{12}=\{e,a^8 b\}, H_{13}=<a^3,b>,\\ H_{14}=<a^3,ab>,  H_{15}=<a^3,a^2b>  $. The following are the hyperedges of $\tilde{\Gamma}_\mathcal{H}(D_9)$:\\
		$e_1=\{H_2,H_4,H_5,H_6,H_7,H_8,H_9,H_{10},H_{11},H_{12}\}$,
		$e_2=\{H_3,H_4,H_5,H_6,H_7,H_8,H_9,H_{10},H_{11},H_{12}\}$,
		$e_3=\{H_5,H_6,H_8,H_9,H_{11},H_{12},H_{13}\}$,
		$e_4=\{H_4,H_6,H_7,H_9,H_{10},H_{12},H_{14}\}$ and 
		$e_5=\{H_4,H_5,H_7,H_8,H_{10},H_{11},H_{15}\}$.\\
		In  the incidence graph  $\mathcal{I}(\tilde{\Gamma}_\mathcal{H}(D_9))$ of $\tilde{\Gamma}_\mathcal{H}(D_9)$, the number of vertices, $n_0 = 19$, the number of edges, $m=41$ and the girth, $gr= 4$ as $H_4e_1H_5e_2H_4$ is a shortest cycle of length 4 in  $\mathcal{I}(\tilde{\Gamma}_\mathcal{H}(D_9))$. \\Thus, the genus $ g(\mathcal{I}(\tilde{\Gamma}_\mathcal{H}(D_9))) \geq \dfrac{m\cdot (gr - 2)}{2 \cdot gr} - \dfrac{n_0}{2} + 1 = \dfrac{41 \cdot 2}{2 \cdot 4} - \dfrac{19}{2} + 1 = 1.75.$
		
		\noindent Consequently, by Theorem \ref{incidencegraph}, $g(\tilde{\Gamma}_\mathcal{H}(D_9))$ is at least 2.

		\noindent $\underline{\textbf{Subcase 2.1(b)}}:$ Suppose $ n= p^\alpha >  9 \, \text{and}\, \alpha = 2, \, \text{i.e.,} \,  p^2 > 9.$
		Consider the subhypergraph $G$ induced by the vertex set $\Omega = \{H_0, H_1,...,H_{p^2 - 1}, <a>, <a^p>, <a^p,b>\},$ where $  H_i = <a^ib>, \,\\ i \in \{0,1,...,2^\alpha -1\}$. So, the hyperedges of the subhypergraph $G$ induced by $\Omega$ are:\\
		$e_1 = \{H_0,H_1,...,H_{p^2-1},<a>\}$,
		$e_2 = \{H_0,H_1,...,H_{p^2 -1},<a^p>\}$ and \\
		$e_3 = \{H_0,H_1,...,H_{p^2 -1},<a^p,b>\}$.  
		In the incidence graph $\mathcal{I}(G)$ of $G$, the number of vertices, $n_0= (p^2 + 3) + 3 = p^2 + 6, $ 
		the number of edges, $m= (p^2 +1)\cdot 2 + p^2 - p + 1 = 3 \cdot p^2-p + 3,$ and 
		the girth, $gr = 4$.
		So, by Lemma \ref{genus} and as $p \geq 5$,
		\begin{align*}
			g(\mathcal{I}(G)) \geq 	\dfrac{m\cdot (gr - 2)}{2 \cdot gr} - \dfrac{n_0}{2} + 1 & =\dfrac{(3 \cdot p^2-p + 3)(4 - 2)}{2 \cdot 4} -\dfrac{p^2 + 6}{2} + 1 \nonumber \\
			&= \dfrac{p(p-1)}{4} -\, 1 .25 \nonumber\\
			& \geq 5 - 1.25 =  3.75.
		\end{align*}
		
		\noindent Consequently, by Theorem \ref{incidencegraph},  $g(\tilde{\Gamma}_\mathcal{H}(D_n)) > 2$.

		\noindent  $\underline{\textbf{Subcase 2.1(c)}}:$ Suppose $n= p^\alpha >  9 \, \text{and}\, \alpha \geq 3.$ 
		Consider the subhypergraph $G$ induced by the vertex set, $\Omega = \{H_0, H_1,...,H_{p^\alpha - 1}, K_0, K_1,K_2\}$, where $  H_i = <a^ib>, \,  i \in \{0,1,...,p^\alpha -1\},$ and $K_j=<a^{p^j}>, \, j \in \{ 0,1,2\},  \,  . $ So, the hyperedges of $G$ are 
		$e_1 = \{H_0,H_1,...,H_{p^{\alpha}-1},K_0\}$,\\
		$e_2 = \{H_0,H_1,...,H_{p^{\alpha}-1},K_1\}$ and
		$e_3 = \{H_0,H_1,...,H_{p^{\alpha}-1},K_2\}$.
		Note that each hyperedge has $p^\alpha + 1 $ vertices. Consider the corresponding incidence graph, $\mathcal{I}(G)$. In $\mathcal{I}(G)$, the number of vertices, $n_0= (p^\alpha + 3) + 3 = p^\alpha + 6 $, 
		the number of edges, $m= (p^\alpha +1)\cdot 3 = 3 \cdot p^\alpha + 3,$ and 
		the girth, $gr = 4$. So, by Lemma \ref{genus} and as $p^\alpha > 9$,
		\begin{align*}
			\dfrac{m\cdot (gr - 2)}{2 \cdot gr} - \dfrac{n_0}{2} + 1 & =\dfrac{(3 \cdot p^\alpha + 3)(4 - 2)}{2 \cdot 4} -\dfrac{p^\alpha + 6}{2} + 1  \\
			& = \dfrac{p^\alpha}{4} - 1.25  > 1.
		\end{align*}
		
		\noindent Consequently, by Theorem \ref{incidencegraph}, $g(\tilde{\Gamma}_\mathcal{H}(D_n))$ is at least 2.
		
		\noindent $\underline{\textbf{Subcase 2.2}}:$ $n$ has at least two distinct prime divisors.

		Let $k$ be the largest integer such that $ k < n$ and $k \mid n $.	Then, $|<a^k>|=p_1$, where $p_1$ is the smallest prime divisor of $n$. Let $G$ be the subhypergraph of $\tilde{\Gamma}_\mathcal{H}(D_n)$	induced by the vertex set 
		$ \Omega=\{H_0,H_1,...,H_{n-1}, <a>, <a^k>, <a^k,b>\}$,
		where $ H_i = <a^ib>, \, i \in \{0,1,...,n -1\}. $
		So, the hyperedges of the subhypergraph $G$ induced by $\Omega$ are 
		$e_1 = \{H_0,H_1,...,H_{n-1},<a>\}$,\\
		$e_2 = \{H_0,H_1,...,H_{n-1},<a^{k}>\}$ and 
		$e_3 = \{H_0,H_1,...,H_{n-1},<a^k,b>\}$.\\
		In  the incidence graph $\mathcal{I}(G)$ of $G$, the number of vertices, $n_0= n + 6, $
		the number of edges, $m= 3n + 3 - p_1,$ and 
		the girth, $gr = 4$. Observe that, for all odd numbers that are not a power of prime, the difference between the number and its smallest prime divisor is strictly greater than 9. So,  by Lemma \ref{genus},	\begin{align*}
			g(G) \geq \dfrac{m\cdot (gr - 2)}{2 \cdot gr} - \dfrac{n_0}{2} + 1 & =\dfrac{(3n + 3 - p_1)(4 - 2)}{2 \cdot 4} -\dfrac{n + 6}{2} + 1 \nonumber \\
			&= \dfrac{n - p_1}{4} -\frac{5}{4} \\
			& >  \dfrac{9}{4} - 1.25 > 1.
		\end{align*}
		
		\noindent Consequently,  by Theorem \ref{incidencegraph}, $g(\tilde{\Gamma}_\mathcal{H}(D_n))$ is at least 2.
		
	\end{proof}
	
	\begin{corollary}
		If $\tilde{\Gamma}_\mathcal{H}(D_n)$ is non-planar, then it cannot be embedded on a torus.
	\end{corollary}
	
	\noindent The following results play a key role in proving the non-orientable genus of $\tilde{\Gamma}_\mathcal{H}(D_n)$.
	\begin{lemma}\cite{bojan}\label{kngenus}
		$\tilde{g}(K_{m,n}) = 
		\left\lceil \dfrac{(m-2)(n-2)}{2} \right\rceil, \,  m,n\geq 2$		
	\end{lemma}

	\begin{lemma}$\cite{bojan}$ \label{nongenus}
		Let $G$ be a connected graph with $n \geq 3$ vertices and
		$q$ edges. If $G$ contains no cycle of length 3, then $\tilde{g}(G) \geq \left\lceil \dfrac{q}{2} - n +2\right\rceil.$
	\end{lemma}
	
	\begin{theorem}\label{incidencegraph1}\cite{walsh1975hypermaps}
		For any hypergraph $\mathcal{H}$, $\tilde{g}(\mathcal{H})=\tilde{g}(\mathcal{I}(\mathcal{H}))$.
	\end{theorem}
	\begin{theorem}
		$\tilde{\Gamma}_\mathcal{H}(D_n)$ is non-planar iff the non-orientable genus,	$\tilde{g}(\tilde{\Gamma}_\mathcal{H}(D_n))$ is at least 2.
	\end{theorem}
	\begin{proof}
		If the non-orientable genus,	$\tilde{g}(\tilde{\Gamma}_\mathcal{H}(D_n))$ is at least 2, then clearly  $\tilde{g}(\tilde{\Gamma}_\mathcal{H}(D_n))$ is non-planar.
		
		Conversely, suppose that $\tilde{\Gamma}_\mathcal{H}(D_n)$ is non-planar. By Theorem \ref{planar},  $\tilde{\Gamma}_\mathcal{H}(D_n)$ is non-planar iff neither $n$ is a prime nor  $ n= 4$. To prove that the non-orientable genus,	$\tilde{g}(\tilde{\Gamma}_\mathcal{H}(D_n))$ is at least 2, consider the following cases:
		
		\noindent $\underline{\textbf{Case 1.}}$ Suppose that $n=6$. 
		The incidence graph $\mathcal{I}(\tilde{\Gamma}_\mathcal{H}(D_6))$ of $\tilde{\Gamma}_\mathcal{H}(D_6)$ does not contain a cycle of length 3. In  $\mathcal{I}(\tilde{\Gamma}_\mathcal{H}(D_6))$,  the number of vertices, $n_0 = 21$, and the number of edges, $q=43$. Thus, by Lemma \ref{nongenus}, $\tilde{g}(\mathcal{I}(\tilde{\Gamma}_\mathcal{H}(D_6))) \geq \left\lceil \dfrac{43}{2} - 21 +2 \right\rceil = 3.$ Consequently, by Theorem \ref{incidencegraph1}, $\tilde{g}(\tilde{\Gamma}_\mathcal{H}(D_6)) > 2$. 
		
		\noindent $\underline{\textbf{Case 2.}}$ Suppose that $n=8$. Consider $K_1 =<a^4>, K_2=<a^2>, K_3=<a>, K_4=<a^4,a^3b>$ and $K_5=<a^2,ab>$. Observe that  $K_i \cap K_j \neq \{e\}$ for $i,j \in \{1,2,3,4,5\}$. Choose $l_1,l_2,l_3$ from the set $ \{0,1,2,\cdots,n-1\}$ such that $l_i \not\equiv 3 (mod \, 4)$ and $l_i \not\equiv 1 (mod \, 2)$ for $ i \in \{1,2,3\}$. Then, for $H_1=<a^{l_1}b>, H_2=<a^{l_2}b>$ and $H_3=<a^{l_3}b>$, we have $H_i \cap H_j =\{e\}$ for distinct $i,j \in \{1,2,3\}$. Also, $H_i \cap K_j =\{e\}$ for $i \in \{1,2,3\}$ and $ j \in \{1,2,3,4,5\}$. Therefore, there exist five distinct hyperdges $e_1,e_2,e_3,e_4,e_5$ such that $e_1$ contains $H_1,H_2,H_3 \,\text{and}\, K_1$, $e_2$ contains $H_1,H_2,H_3 \,\text{and}\, K_2$, $e_3$ contains $H_1,H_2,H_3 \,\text{and}\, K_3$, $e_4$ contains $H_1,H_2,H_3 \,\text{and}\, K_4$, and $e_5$ contains $H_1,H_2,H_3 \,\text{and}\, K_5$. Consider the  subhypergraph $G_1$ of $\tilde{\Gamma}_\mathcal{H}(D_8)$ with  $  \{e_1,e_2,e_3,e_4,e_5\}$ as the  hyperedge set and  the set of all vertices in $e_1,e_2,e_3,e_4\, \text{and} \,\, e_5$ as the vertex set of $G_1$. Thus, the incidence graph $\mathcal{I}(G_1)$ of $G_1$ contains $K_{3,5}$ as a subgraph. By Lemma \ref{kngenus}, $\tilde{g}(K_{3,5}) = 2$. Consequently, by Theorem \ref{incidencegraph1},  $\tilde{g}(\tilde{\Gamma}_\mathcal{H}(D_8)) \geq 2.$

		
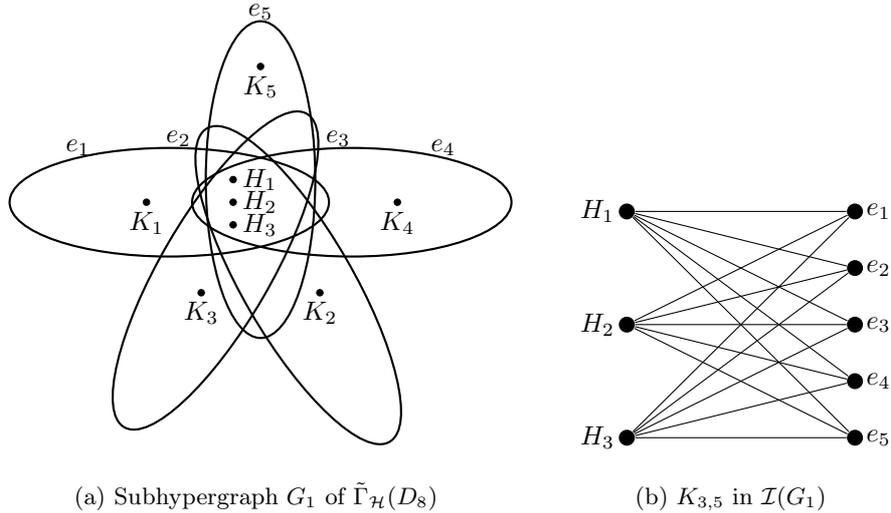
\begin{figure}[h]
			\centering\subfloat[Subhypergraph $G_1$ of $\tilde{\Gamma}_\mathcal{H}(D_8)$]{	\begin{tikzpicture}[scale=.6]
					
					\draw[thick] (-2,0) ellipse (3.5 and 1.2);
					\draw[thick] (2,0) ellipse (3.5 and 1.2);
					\draw[rotate=330,thick] (-.1,-1.8) ellipse (1.2 and 4);
					\draw[rotate=30,thick] (-.2,-2) ellipse (1.2 and 4);
					\draw[thick] (0,.5) ellipse (1.2 and 3.5);
					
					\node at (-4, 1.2) {$e_1$};
					\node at (-1.8, 1.4) {$e_2$};
					\node at (1.7, 1.4) {$e_3$};
					\node at (4, 1.2) {$e_4$};
					\node at (0, 4.2) {$e_5$};
					
					
					\filldraw (-0.6, 0.5cm) circle (2pt) node[right] {${H_1}$};
					\filldraw (-0.6, 0cm) circle (2pt) node[right] {${H_2}$};
					\filldraw (-0.6, -.5cm) circle (2pt) node[right] {${H_3}$};
					\filldraw (1.3, -2cm) circle (2pt) node[below] {${K_2}$};
					
					\filldraw (-2.5, 0cm) circle (2pt) node[below] {${K_1}$};
					\filldraw (-1.3, -2cm) circle (2pt) node[below] {${K_3}$};
					\filldraw (3, 0cm) circle (2pt) node[below] {${K_4}$};
					\filldraw (0, 3cm) circle (2pt) node[below] {${K_5}$};
					
			\end{tikzpicture}}
			\label{G2}
			\hspace{.50 cm}
			\subfloat[$K_{3,5}$ in $\mathcal{I}(G_1)$]{
				\begin{tikzpicture}[scale=1.5, every node/.style={circle, fill=black, minimum size=6pt, inner sep=0pt}]
					
					\node[label=left:$H_1$] (u1) at (0,3) {};
					\node[label=left:$H_2$] (u2) at (0,2) {};
					\node[label=left:$H_3$] (u3) at (0,1) {};
					
					\node[label=right:$e_1$] (v1) at (2,3) {};
					\node[label=right:$e_2$] (v2) at (2,2.5) {};
					\node[label=right:$e_3$] (v3) at (2,2) {};
					\node[label=right:$e_4$] (v4) at (2,1.5) {};
					\node[label=right:$e_5$] (v5) at (2,1) {};
					
					\draw (u1) -- (v1);
					\draw (u1) -- (v2);
					\draw (u1) -- (v3);
					\draw (u1) -- (v4);
					\draw (u1) -- (v5);
					\draw (u2) -- (v1);
					\draw (u2) -- (v2);
					\draw (u2) -- (v3);
					\draw (u2) -- (v4);
					\draw (u2) -- (v5);
					\draw (u3) -- (v3);
					\draw (u3) -- (v2);
					\draw (u3) -- (v1);
					\draw (u3) -- (v4);
					\draw (u3) -- (v5);

			\end{tikzpicture}}
			
			\label{k351}

		\end{figure}

		\noindent $\underline{\textbf{Case 3.}}$ Suppose that $n \geq 9.$ 
	
		\begin{figure}[h]
			\centering\subfloat[Subhypergraph $G_2$ of $\tilde{\Gamma}_\mathcal{H}(D_n)$]{	\begin{tikzpicture}[scale=.5]
					
					\draw[thick] (-2,0) ellipse (4 and 2);
					\draw[thick] (2,0) ellipse (4 and 2);
					\draw[thick] (0,-2) ellipse (2 and 4);
					
					\node at (-4, 2) {$e_1$};
					\node at (0, 2.2) {$e_2$};
					\node at (4, 2) {$e_3$};
					
					
					\filldraw (-1.2, .5cm) circle (2pt) node[right] {${H_1}$};
					\filldraw (0.2, .5cm) circle (2pt) node[right] {${H_2}$};
					\filldraw (-0.5, -.2cm) circle (2pt) node[right] {${H_3}$};
					\filldraw (-1.2, -.9cm) circle (2pt) node[right] {${H_4}$};
					\filldraw (0.2, -.9cm) circle (2pt) node[right] {${H_5}$};
					\filldraw (0.2, -2.5cm) circle (2pt) node[below] {${K_2}$};
					\filldraw (-2.5, 0cm) circle (2pt) node[below] {${K_1}$};
					\filldraw (2.5, 0cm) circle (2pt) node[below] {${K_3}$};
					
			\end{tikzpicture}}
			\label{G1}
			\hspace{.1 cm}
			\subfloat[$K_{3,5}$ in $\mathcal{I}(G_2)$]{
				\begin{tikzpicture}[scale=1.5, every node/.style={circle, fill=black, minimum size=6pt, inner sep=0pt}]
					
					\node[label=right:$H_1$] (u1) at (2,3) {};
					\node[label=right:$H_2$] (u2) at (2,2.5) {};
					\node[label=right:$H_3$] (u3) at (2,2) {};
					\node[label=right:$H_4$] (u4) at (2,1.5) {};
					\node[label=right:$H_5$] (u5) at (2,1) {};
					
					\node[label=left:$e_1$] (v1) at (0,3) {};
					\node[label=left:$e_2$] (v2) at (0,2) {};
					\node[label=left:$e_3$] (v3) at (0,1) {};
					
					\draw (u1) -- (v1);
					\draw (u1) -- (v2);
					\draw (u1) -- (v3);
					\draw (u2) -- (v1);
					\draw (u2) -- (v2);
					\draw (u2) -- (v3);
					\draw (u3) -- (v3);
					\draw (u3) -- (v2);
					\draw (u3) -- (v1);
					\draw (u4) -- (v1);
					\draw (u4) -- (v2);
					\draw (u4) -- (v3);
					\draw (u5) -- (v1);
					\draw (u5) -- (v2);
					\draw (u5) -- (v3);
					
			\end{tikzpicture}}
			
			\label{k35}

		\end{figure}
			Consider $K_1=<a>, K_2=<a^p>$ and $K_3 = <a^p,b>$, where $p$ is a prime divisor of $n$. Observe that $K_i \cap K_j \neq \{e\}$ for $i,j \in \{1,2,3\}$. Choose $l_1,l_2,l_3,l_4$ and $l_5$ from the set $\{0,1,2,\cdots,n-1\}$ such that $l_i \not\equiv 0 (mod \, p)$ for all $i \in \{1,2,3,4,5\}$. Then, for $H_1=<a^{l_1}b>,$ $ H_2=<a^{l_2}b>,$ \\$ H_3=<a^{l_3}b>, H_4=<a^{l_4}b>$ and $H_5=<a^{l_5}b>$, we have $H_i \cap H_j =\{e\}$ for distinct $i,j \in \{1,2,3\}$. Also, $H_i \cap K_j =\{e\}$ for $i \in \{1,2,3,4,5\}$ and $j \in \{1,2,3\}$. Therefore, there exist three distinct hyperedges $e_1,e_2,e_3$ such that $e_1$ contains $H_1,H_2,H_3,H_4,H_5 \, \text{and}\, K_1$, $e_2$ contains $H_1,H_2,H_3,H_4,H_5 \, \text{and}\, K_2$, and $e_3$ contains $H_1,H_2,H_3,H_4,H_5 \, \text{and}\, K_3$. 		Consider the  subhypergraph $G_2$ of $\tilde{\Gamma}_\mathcal{H}(D_n)$ with  $  \{e_1,e_2,e_3\}$ as the  hyperedge set and  the set of all vertices in $e_1,e_2\, \text{and} \,\, e_3$ as the vertex set of $G_2$. Thus, the incidence graph $\mathcal{I}(G_2)$ of $G_2$ contains $K_{3,5}$ as a subgraph.  By Lemma \ref{kngenus}, $\tilde{g}(K_{3,5}) = 2$. Consequently, by Theorem \ref{incidencegraph1},  $\tilde{g}(\tilde{\Gamma}_\mathcal{H}(D_n)) \geq 2.$	
	\end{proof}
	
	\begin{corollary}
		If $\tilde{\Gamma}_\mathcal{H}(D_n)$ is non-planar, then it cannot be embedded on a projective plane.
	\end{corollary}
	\begin{corollary}
		$\tilde{\Gamma}_\mathcal{H}(D_n)$ is either planar or has both orientable and non-orientable genus at least two.
	\end{corollary}

	\section*{Acknowledgments}
	This research work is an outcome of the project supported by the Institute of Eminence (UoH-IoE-RC5-22-021), University of Hyderabad.

\end{document}